\documentclass[11pt, reqno]{amsart}
\usepackage{amsmath, amsthm, amssymb, amsfonts, mathrsfs}

\usepackage[margin=3.5cm]{geometry}

\usepackage{color}
\usepackage[pagebackref=true]{hyperref}

\numberwithin{equation}{section}
\newcommand{\QQ}{\mathbb{Q}}
\newcommand{\ZZ}{\mathbb{Z}}

\newcommand{\xx}{\mathbf{x}}
\newcommand{\yy}{\mathbf{y}}
\DeclareMathOperator{\even}{even}
\DeclareMathOperator{\odd}{odd}
\renewcommand{\leq}{\leqslant}
\renewcommand{\geq}{\geqslant}

\title{Random Diophantine equations of large degree}

\author{Tim Browning}
\address{ISTA \\ Am Campus 1 \\ 3400 Klosterneuburg \\ Austria}
\email{tdb@ist.ac.at}

\author{Will Sawin}
\address{Princeton University\\
Fine Hall\\ 304 Washington Rd\\
Princeton NJ 08540\\ USA}\email{wsawin@math.princeton.edu}

\subjclass[2010]{11D45 (11G50, 11P21, 14G05)}
\keywords{Hasse principle, Fano hypersurfaces, rational points, heights}

\begin{document}

\begin{abstract}
Among the set of hypersurfaces of degree $d$ and dimension $\ell$ defined by the vanishing of a homogeneous polynomial with coefficients in the set $\{-1, 1\}$, we investigate the  probability that a hypersurface contains a rational point as $d$ and $\ell$ tend to $\infty$.
\end{abstract}

\maketitle

\setcounter{tocdepth}{1}
\tableofcontents

\thispagestyle{empty}

\newtheorem{theorem}{Theorem}[section]
\newtheorem{lemma}[theorem]{Lemma}
\newtheorem{con}[theorem]{Conjecture}

\theoremstyle{definition}
\newtheorem*{ack}{Acknowledgements}
\newtheorem{definition}[theorem]{Definition}

\section{Introduction}

Questions about the arithmetic properties of randomly chosen polynomials with integer coefficients
are of central concern in number theory. 
In the simplest model, the {\em large box model}, one fixes the degree of the polynomial 
and allows the coefficients to vary uniformly in $\{-H,\dots,H\}$, as $H\to \infty$.
In the {\em large degree model}, on the other hand, one constrains the coefficients to be in 
 $\{-1,+1\}$, say,  and allows the degree of the polynomial to grow. 
It has been known for a long time that  random polynomials $f\in \mathbb{Z}[x]$ 
are irreducible over $\mathbb{Q}$, when adhering to the large box model.  The best known error term is due to Kuba \cite{kuba}, who  shows that the probability of $f$ being reducible is $O_d(H^{-1})$, for fixed degree $d=\deg(f)\geq 3$. The  large degree model is significantly harder to analyse, and we only have a conditional treatment. Thus, 
under the assumption of the Riemann Hypothesis  for Dedekind zeta functions, 
the question of irreducibility in  the large degree model has been dealt with  by Breuillard and Varj\'u \cite{irred}.

In this paper we shall be concerned with the Hasse principle for 
hypersurfaces $V\subset \mathbb{P}^{n}$ defined by 
randomly chosen homogeneous 
polynomials 
$f\in \mathbb{Z}[x_0,\dots,x_n]$ of degree $d$. 
In the large box model, 
this topic was 
investigated by Poonen and Voloch~\cite{PV}.  In 
\cite[Conjecture 2.2.(ii)]{PV} it is asserted that in the Fano regime  $n \geq d$, $100\%$ of the hypersurfaces should satisfy the Hasse principle.  For the range
 $n < d-1$, corresponding to  hypersurfaces of general type, 
it is predicted in \cite[Conjecture 2.2.(i)]{PV} that $0\%$ of the hypersurfaces should have a rational point.
The first  of these conjectures is settled for  hypersurfaces of dimension at least $3$
in joint work of the authors with Le Boudec  \cite{random}.

We shall address the existence of rational points  in  the large degree model. 
Let $d, n \geq 2$ and let $N_{d,n} = \binom{n+d}{d}$ be the number of monomials of degree $d$ in $n+1$ variables. Let  $\mathbb{B}_{d,n}$ be the set of hypersurfaces in $\mathbb{P}^n$ defined by the vanishing of a homogeneous polynomial of degree $d$ with coefficients in the set $\{-1, 1\}$. We thus have
\begin{equation}
\label{Cardinality Bdn}
\# \mathbb{B}_{d,n} = 2^{N_{d,n}-1}.
\end{equation}
We shall be interested in the ratio
\begin{equation*}
r_{d,n} = \frac{\# \{ V \in \mathbb{B}_{d,n} : V(\mathbb{Q}) \neq \emptyset \}}{\# \mathbb{B}_{d,n}},
\end{equation*}
which  measures the density of  hypersurfaces $V\in \mathbb{B}_{d,n} $ with rational points. 
The following result describes the size  of $r_{d,n}$, as $d\to \infty$, and 
is rather close to covering  the entire  Fano range.

\begin{theorem}
\label{Theorem all d}
For $n \geq d+\log d$, we have
\begin{equation*}
r_{d,n} = 1+ O \left( \frac1{d^{1/4}} \right),
\end{equation*}
where the implied constant is absolute.
\end{theorem}

Our next result shows that if we are ready to allow the integer $d$ to avoid an appropriate 
density $0$ set of integers, then the condition on $n$ can be  relaxed dramatically.

\begin{theorem}
\label{Theorem almost all d}
There exists a set $\mathcal{D}$ of density $1$ in the positive integers, such that for $d \in \mathcal{D}$ and 
$n \geq d^{1/2}\log d$, we have
\begin{equation*}
r_{d,n} = 1 + O \left( \frac{1}{\sqrt{d}}\right),
\end{equation*}
where the implied constant is absolute.
\end{theorem}

It follows that in the range $n \geq d^{1/2}\log d$, the proportion of hypersurfaces in $\mathbb{B}_{d,n}$ which have rational points  approaches $100\%$, as $d \to \infty$ while running over an appropriate set of density $1$. It is rather  striking to see that so much of the general type range is handled in this result.  The most troublesome degrees  $d$ to handle are those whose  $2$-adic valuations are large and the set  $\mathcal{D}$ is chosen to avoid these.

Returning to the question of irreducibility, it follows 
from 
work of Bary-Soroker and Kozma \cite{bary} that 
the proportion of 
ternary degree $d$ forms 
that are irreducible over $\mathbb{Q}$
approaches  $100\%$ as $d\to \infty$. 
A form $f\in \mathbb{Z}[x_0,\dots,x_n]$ is irreducible over $\mathbb{Q}$ if 
$f(x_0,x_1,x_2,0,\dots,0)$ has this property. 
It therefore  follows that 
the proportion of 
hypersurfaces in  $\mathbb{B}_{d,n}$ that are irreducible over $\mathbb{Q}$
approaches  $100\%$ as $d\to \infty$, 
provided that $n\geq 2$. 
Unlike the case of univariate polynomials, this result is  unconditional. 
On the other hand, while the large box model version of the following conjecture is straightforward, 
it would be interesting to determine its validity in the large degree model.

\begin{con}
100\% of hypersurfaces $V\in \mathbb{B}_{d,n}$ are smooth, as $d\to \infty$.
\end{con}

The proof of Theorems \ref{Theorem all d} and 
\ref{Theorem almost all d} is rather combinatorial and will be completed in Section \ref{s:proofs}.
The key idea is
to study the moments 
$$
M_k(U) = \frac1{\# \mathbb{B}_{d,n}} \sum_{V \in \mathbb{B}_{d,n}} \left( \sum_{x \in U\cap V(\mathbb{Q})} 1 \right)^k,
$$
for $k\in \{1,2\}$, where $U$ is an appropriately chosen set of points in $\mathbb{P}^{n}(\mathbb{Q})$, always with coordinates drawn from the set $\{-1,0,+1\}$. One accesses $r_{d,n}$ through the Cauchy--Schwarz inequality, which yields
$M_1(U)^2\leq r_{d,n} M_2(U)$.  Providing estimates for the moments $M_k(U)$ is  the focus of Section \ref{sec:moments}.
Along the way we will need information about the density of locally soluble hypersurfaces in the large degree model, which is the focus of Section \ref{s:local}, in addition to  various estimates for binomial coefficients and sums involving them, for which the relevant facts are collected together in Section \ref{sec:binom}.

\begin{ack}
The authors are very grateful to Pierre Le Boudec, whose contributions would have merited co-authorship. At his request, however, he is not listed as a co-author, since he has chosen to pursue a career outside academia.
T.B.\ was
supported by  a grant from the IAS School of Mathematics and a FWF grant (DOI 10.55776/P36278).
W.S.\ was supported by  NSF grant DMS-2502029 and a Sloan Research Fellowship.

\end{ack}

\section{Local solubility}\label{s:local}

In this section we investigate the ratio
\begin{equation*}
r_{d,n}^{\mathrm{loc}} = \frac{\# \{ V \in \mathbb{B}_{d,n} : V(\mathbf{A}_{\mathbb{Q}}) \neq \emptyset \}}{\# \mathbb{B}_{d,n}},
\end{equation*}
where $\mathbf{A}_{\mathbb{Q}}$ is the ring of ad\`eles. This is  the probability that a hypersurface in $\mathbb{B}_{d,n}$ is everywhere locally soluble.  We begin with the following elementary fact about the parity of 
$N_{e,\ell} = \binom{e+\ell}{e}$. 

\begin{lemma}\label{lem:N-even}
$N_{e,\ell}$ is even if and only if 
$e$ and $\ell$ share a common digit $1$ in their binary expansions.
\end{lemma}

\begin{proof}
Given any integer $e \geq 1$, we let $s_2(e)$ denote the sum of the digits in the binary expansion of $e$. Let  $v_2(m)$ denote the $2$-adic valuation of any integer $m \geq 1$. Legendre's formula states that 
\begin{equation*}
v_2(e!) = e - s_2(e).
\end{equation*}
It follows that for any $e, \ell \geq 1$, we have
\begin{equation}
\label{Equality valuation}
v_2 ( N_{e,\ell} ) = s_2(e) + s_2(\ell) - s_2(e+\ell).
\end{equation}
As a result, we see that $N_{e,\ell}$ is even if and only if
\begin{equation*}
s_2(e) + s_2(\ell) > s_2(e+\ell),
\end{equation*}
which is equivalent to demanding that $e$ and $\ell$ share a common digit $1$ in their binary expansions.
\end{proof}

It will be convenient to have an alternative formulation of 
$N_{d,\ell}$, for any  $d, \ell \geq 1$. Let  $j \in \{1, \dots, \ell\}$. Then the number of monomials of degree $d$ in $\ell+1$ variables is equal to the sum over $k \in \{0, \dots, d\}$ of the number of monomials of degree $k$ in $j$ given variables multiplied by the number of monomials of degree $d-k$ in the remaining $\ell+1-j$ variables. In other words, we have
\begin{equation}
\label{Equality number monomials}
N_{d,\ell} = \sum_{k=0}^d \binom{k+j-1}{k} \binom{d-k+\ell-j}{d-k}.
\end{equation}
We now let $c_{d,\ell}(j)$ denote the number of monomials of degree $d$ which are odd with respect to $j$ given variables among $\ell+1$. Thus 
\begin{equation}
\label{Definition c}
c_{d,\ell}(j) = \sum_{\substack{k=0 \\ 2 \nmid k}}^d \binom{k+j-1}{k} \binom{d-k+\ell-j}{d-k}.
\end{equation}

We proceed by 
proving an  upper bound for 
$r_{d,n}^{\mathrm{loc}}$, which may prove to be of independent interest. 
It would be interesting to determine whether or not these upper bounds have matching lower bounds 
when $n\geq 2^{v_2(d)}$.

\begin{lemma}
Let $d,n \geq 2$. Then
\begin{equation*}
r_{d,n}^{\mathrm{loc}} \leq 
\begin{cases}
0 & \textrm{if } n < 2^{v_2(d)}, \\
1/2 & \textrm{if } n = 2^{v_2(d)}, \\
1 & \textrm{if } n > 2^{v_2(d)}.
\end{cases}
\end{equation*}
\end{lemma}

\begin{proof}
Let $d\geq 2$. 
It follows from Lemma 
\ref{lem:N-even} that  the least $n \geq 2$ such that $N_{d,n}$ is even is $n = 2^{v_2(d)}$.
Thus,  $N_{d,m}$ is odd for any  $m \in \{1, \dots, n \}$ if $n < 2^{v_2(d)}$. 
It follows that $V(\mathbb{Q}_2) = \emptyset$ for  any $V \in \mathbb{B}_{d,n}$, whence in 
fact $r _{d,n}^{\mathrm{loc}} = 0$.
If $n > 2^{v_2(d)}$ then we take the trivial  upper bound 
 $r _{d,n}^{\mathrm{loc}} \leq 1$.
 
Suppose now that  $n = 2^{v_2(d)}$. Since $n\geq 2$, it follows that both $n$ and $d$ are even. 
For the purposes of an upper bound it will suffice to consider the proportion of $V \in \mathbb{B}_{d,n}$ which admit points modulo $4$.  Let $f\in \mathbb{Z}[x_0,\dots,x_n]$ be a form of degree $d$ that defines $V$, which has coefficients in the set $\{-1,+1\}$.  
Thus there exist  $c_{i_0,\dots,i_n}\in \{-1,+1\}$ such that 
$$
f(x_0,\dots,x_n)=\sum_{\substack{
i_0,\dots,i_n\geq 0\\
i_0+\dots+i_n=d
}}c_{i_0,\dots,i_n} x_0^{i_0}\dots x_n^{i_n}.
$$
Since $N_{d,m}$ is odd for any $m<n$, it follows that $(1:1:\cdots:1)$ is the only  possible solution modulo  $2$.  Thus the  only points in $V(\mathbb{Q}_2)$ have representatives in which all coordinates are congruent to $-1$ or $+1$ modulo $4$, with at least one coordinate equal to $+1$ modulo $4$. 
We claim that 
$$
f(x_0,\dots,x_n)\equiv  f(1,\dots,1) \bmod{4},
$$
for any such point $(x_0:\cdots:x_n)\in V(\mathbb{Q}_2)$. To see this we may suppose  without loss of generality that  $x_0=\dots=x_{j-1}=-1$ and $x_{j}=\dots=x_n=+1$ for some $1\leq j\leq n$.
Then 
the claim is equivalent to 
$$
\sum_{\substack{
i_0,\dots,i_n\geq 0\\
i_0+\dots+i_n=d
}}c_{i_0,\dots,i_n} \left( 1-(-1)^{i_0+\dots+i_{j-1}} \right) \equiv 0\bmod{4},
$$
which in turn is equivalent to the condition 
$$
\sum_{\substack{
i_0,\dots,i_n\geq 0\\
i_0+\dots+i_{j-1}\equiv 1 \bmod{2}\\
i_0+\dots+i_n=d
}}c_{i_0,\dots,i_n}  \equiv 0\bmod{2}.
$$
Since $c_{i_0,\dots,i_n}\equiv 1 \bmod{2}$ for any choice of indices $i_0,\dots,i_n$, we see that 
the left hand side is congruent modulo $2$ to the number $c_{d,n}(j)$
of monomials of degree $d$ which are odd with respect to $j$ given variables among $n+1$. 
Hence, in the light of  \eqref{Definition c}, the condition is equivalent to demanding that 
$$
\sum_{\substack{k=0 \\ 2 \nmid k}}^d \binom{k+j-1}{k} \binom{d+n-j-k}{d-k} 
$$
is even.  
Suppose first that $j$ is even. Then $j-1$ is odd and  $\binom{k+j-1}{k}$ is even for all odd $k$, by Lemma 
\ref{lem:N-even}. Alternatively, if $j$ is odd then
$n-j$ is odd, since $n$ is even. It therefore follows from Lemma 
\ref{lem:N-even} that 
$\binom{d+n-j-k}{d-k} $ 
is  even for all odd $k$, since  $d$ is even. 
Now that we have established the claim, the
 statement of the lemma follows on noting that 50\% of hypersurfaces 
$V\in \mathbb{B}_{d,n}$ are defined by forms for which 
$f(1,\dots,1)\equiv 0 \bmod{4}$.
\end{proof}

\section{Sums of binomial coefficients}\label{sec:binom}

Stirling's formula   states that
\begin{equation}
\label{Stirling}
m! \sim (2 \pi)^{1/2} m^{m+1/2} e^{-m},
\end{equation}
as $m\to \infty$, which can be used to prove that 
\begin{equation}
\label{Estimate Stirling binom}
\binom{2m}{m} = \frac{2^{2m}}{\pi^{1/2} m^{1/2}} \left(1 + O \left( \frac1{m} \right) \right).
\end{equation}
Coupled with this, we have the upper bounds
\begin{equation}
\label{Upper bound Stirling binom}
\binom{2m}{m} \leq \frac{2^{2m}}{\pi^{1/2} m^{1/2}} \quad \text{ and } \quad 
\binom{m}{\lfloor m/2 \rfloor} \leq \frac{2^{m+1/2}}{\pi^{1/2} m^{1/2}},
\end{equation}
which are valid for any  $m \geq 1$ and can be established using induction. 

We will always adhere  to the traditional convention that 
$
\binom{r}{s} = 0,
$
for any $r \geq 0$ and $s \notin \{0, \dots, r\}$.  We  make rather heavy use of  Vandermonde convolution identities in our work. 
The first states that 
\begin{equation}
\label{Identity VDM}
\binom{r_1+r_2}{q} = \sum_{\alpha_1=0}^{r_1} \binom{r_1}{\alpha_1} \binom{r_2}{q-\alpha_1},
\end{equation}
for any $r_1, r_2, q \geq 0$. This is proved by considering the coefficients of $x^q$ in the binomial expansion of the identity $(1+x)^{r_1}(1+x)^{r_2}=(1+x)^{r_1+r_2}$. By instead considering  the coefficients of $x^{r_1+q}$ and appealing to the identity
$\binom{r_1}{a_1}=\binom{r_1}{r_1-a_1}$, one  arrives at the companion equality
\begin{equation}
\label{Identity VDM'}
\binom{r_1+r_2}{r_2-q} = \sum_{\alpha_1=0}^{r_1} \binom{r_1}{\alpha_1} \binom{r_2}{q+\alpha_1},
\end{equation}
for any $r_1,r_2\geq 0$ and $q\in \ZZ$.
We shall also need the higher convolution
\begin{equation}
\label{Identity VDM2}
\binom{r_1+r_2+r_3}{q} = \sum_{\alpha_1=0}^{r_1} \ \sum_{\alpha_2=0}^{r_2} \binom{r_1}{\alpha_1} \binom{r_2}{\alpha_2} \binom{r_3}{q-\alpha_1-\alpha_2},
\end{equation}
which is valid for any $r_1,r_2,r_3,q\geq 0$.

It will be convenient to let $D(0,1)$ denote the set of complex numbers with modulus at most $1$. We proceed by proving the following elementary result. 

\begin{lemma}
\label{Lemma series}
Let $h \geq 0$ and $z \in D(0,1)$. We have
\begin{equation*}
\sum_{d=0}^{\infty} \binom{d+h}{d} z^d = \frac1{(1-z)^{h+1}}.
\end{equation*}
\end{lemma}

\begin{proof}
We proceed by induction on $h$. The result holds in the case $h=0$, on executing the geometric series. 
We now assume that the result holds for some $h \geq 0$ and we note that the particular case $j = \ell = h+1$ of the equality \eqref{Equality number monomials} states that 
\begin{equation*}
\binom{d+h+1}{d} = \sum_{k=0}^d \binom{k+h}{k},
\end{equation*}
for any $d,h\geq 0$.
It follows that
\begin{align*}
\sum_{d=0}^{\infty} \binom{d+h+1}{d} z^d 
&=\sum_{d=0}^\infty \sum_{k=0}^{d} \binom{k+h}{k}z^d\\
&=\sum_{k=0}^\infty \binom{k+h}{k}
\sum_{d\geq k}
z^d\\
&=\frac{1}{1-z}\sum_{k=0}^\infty \binom{k+h}{k}
z^k.
\end{align*}
The proof follows on using the induction hypothesis.
\end{proof}

On comparing \eqref{Equality number monomials} and \eqref{Definition c}, 
it is natural to expect that the number $c_{d,\ell}(j)$ should usually be close to $N_{d,\ell}/2$. The  purpose of the following result is to show that this is indeed the case for most $j \in \{1, \dots, \ell\}$, provided 
that  $\ell \leq d$.

\begin{lemma}
\label{Lemma coefficients}
Let $1\leq j\leq \ell \leq d$. Then we have 
\begin{equation*}
\left|1- \frac{2c_{d,\ell}(j)}{N_{d,\ell}}\right|  \leq \left( \frac{1}{2} \right)^{\min \{ j, \ell+1-j \}}.
\end{equation*}
\end{lemma}

\begin{proof}
Recalling \eqref{Equality number monomials}
and \eqref{Definition c}, 
we begin by noting that
\begin{equation}
\label{Equality Cauchy}
N_{d,\ell} - 2 c_{d,\ell}(j) = \sum_{k=0}^d (-1)^k \binom{k+j-1}{k} \binom{d-k+\ell-j}{d-k}.
\end{equation}
Therefore, we see that
\begin{equation*}
N_{d,\ell} - 2 c_{d,\ell}(j) = (-1)^d \left( N_{d,\ell} - 2 c_{d,\ell}(\ell+1-j) \right).
\end{equation*}
Hence, we can assume from now on that $j \leq  (\ell+1)/2$ and we aim to prove that 
\begin{equation}
\label{Upper bound goal}
\left| 1-\frac{2c_{d,\ell}(j)}{N_{d,\ell}} \right| \leq 
 \frac{1}{2^j}.
\end{equation}

It follows from  \eqref{Equality Cauchy} that for any $z \in D(0,1)$ we have
\begin{equation*}
\sum_{d=0}^{\infty} \left( N_{d,\ell} - 2 c_{d,\ell}(j) \right) z^d = \left( \sum_{d=0}^{\infty} (-1)^d \binom{d+j-1}{d} z^d \right) \cdot \left( \sum_{d=0}^{\infty} \binom{d+\ell-j}{d} z^d \right).
\end{equation*}
As a result, Lemma~\ref{Lemma series} gives
\begin{equation*}
\sum_{d=0}^{\infty} \left( N_{d,\ell} - 2 c_{d,\ell}(j) \right) z^d = \frac1{(1+z)^j (1-z)^{\ell-j+1}}
= \frac1{(1-z^2)^j (1-z)^{\ell-2j+1}}.
\end{equation*}
The coefficients of the power series $1/(1-z^2)$ and $1/(1-z)$ are all non-negative, which implies that 
$
N_{d,\ell} - 2 c_{d,\ell}(j) \geq 0,
$ 
for all $d\geq 0$. Note that the coefficients of $1/(1-z^2)$ are all bounded by the coefficients of $1/(1-z)$.
Suppose we have power series
$$
A(z)=\sum_{d=0}^\infty A_d z^d, \quad 
B(z)=\sum_{d=0}^\infty B_d z^d, \quad
C(z)=\sum_{d=0}^\infty C_d z^d, \quad 
D(z)=\sum_{d=0}^\infty D_d z^d,
$$
where $A_d,B_d,C_d,D_d\geq 0$ for all $d\geq 0$.  If $A_d\leq C_d$ and $B_d\leq D_d$ for all $d\geq 0$ then clearly the coefficients of $A(z)B(z)$ are all bounded by the coefficients of $C(z)D(z)$. In this way we deduce that the coefficients of 
$$
 \frac1{(1-z^2)^j (1-z)^{\ell-2j+1}}
$$
are all bounded by the coefficients of 
$$
 \frac1{(1-z)^j (1-z)^{\ell-2j+1}}=
  \frac1{(1-z)^{\ell-j+1}}.
$$
On appealing to Lemma \ref{Lemma series}, we conclude  that 
$$
0\leq N_{d,\ell} - 2 c_{d,\ell}(j)\leq 
\binom{d+\ell-j}{d},
$$
for all $d\geq 0$.
Finally, since $j\leq \ell\leq d$, we note that 
$$
\frac{1}{N_{d,\ell}}\binom{d+\ell-j}{d}=\frac{\ell}{d+\ell}\cdot 
\frac{\ell-1}{d+\ell-1}\cdots 
\frac{\ell-(j-1)}{d+\ell-(j-1)}\leq \frac{1}{2^{j}},
$$
as claimed in 
\eqref{Upper bound goal}.
\end{proof}

We close with the following inequality that exploits a connection to the hypergeometric distribution in probability. 

\begin{lemma}\label{lem:prob}
Let $r,s\geq 0$ be integers. Then 
$$
\sum_{j=0}^r\binom{r}{j}\binom{s}{j} 2^{-j}\leq \left(\frac{3}{4}\right)^\mu \binom{r+s}{r},
$$
where $\mu=rs/(r+s).$
\end{lemma}

\begin{proof}
There is a convenient combinatorial model for this situation. 
Suppose we have a population of size $r+s$, comprising $r$ copies of $0$ and $s$ copies of $-\log 2$. Suppose we pick $r$ elements without replacement. Then the  probability of drawing the value $-\log 2$ exactly $j$ times is given by the hypergeometric distribution 
$$
q_j=
\frac{\binom{r}{j}\binom{s}{j}}{\binom{r+s}{r}}.
$$
Note that the sum of the $r$ samples is equal to $-j\log2=\log(2^{-j})$. Thus the sum
$$
S=\frac{1}{\binom{r+s}{r}}\sum_{j=0}^r\binom{r}{j}\binom{s}{j} 2^{-j}
$$
is equal to the expectation of the exponential of the sum of $r$ samples without replacement. 
It follows from a result of 
Hoeffding \cite[Thm.~4]{hoeffding} that this expectation is bounded 
by the expectation of the exponential of the sum of $r$ samples with replacement, which is equal to the $r$th power of the expectation of one sample. But one sample is either $0$ (with probability $\frac{r}{r+s}$) or
$-\log 2$ (with probability $\frac{s}{r+s}$). It therefore follows that 
$$
S\leq \left(\frac{r}{r+s}\cdot \exp(0) +\frac{s}{r+s}\cdot \exp(-\log 2)\right)^r=\left(1-\frac{s}{2(r+s)}\right)^r\leq
\exp\left(-\frac{rs}{2(r+s)}\right).
$$
The statement of the lemma follows on noting that 
$e^{-1/2}\leq 3/4$.
\end{proof}

\section{The first and second moments}\label{sec:moments}

For $n \geq 2$ and $\ell \in \{1, \dots, n\}$, define  $U_{n,\ell}\subset \mathbb{P}^n(\mathbb{Q})$ 
to be the 
 subset
whose elements have a representative with exactly $\ell+1$ coordinates equal to $\pm 1$ and all remaining coordinates equal to $0$. We clearly have
\begin{equation}
\label{Cardinality Unl}
\# U_{n,\ell} = 2^{\ell} \binom{n+1}{\ell+1}.
\end{equation}
For $k \geq 0$ and any $\ell \in \{1, \dots, n\}$, 
we introduce the moments
\begin{equation}\label{eq:moment}
M_k(d,n;\ell) = \frac1{\# \mathbb{B}_{d,n}} \sum_{V \in \mathbb{B}_{d,n}} \left( \sum_{x \in U_{n,\ell} \cap V(\mathbb{Q})} 1 \right)^k.
\end{equation}
The following  notation will prove useful.

\begin{definition}\label{def-f}
Let $n\geq 1$. Given $f\in \ZZ[x_0,\dots,x_n]$, we let $\nu(f)$ denote the number of coefficients of $f$ equal to $-1$.
\end{definition}

We start by establishing an explicit formula for the first moment.

\begin{lemma}
\label{Lemma first moment}
Let $d, n \geq 2$ and let $\ell \in \{1, \dots, n\}$ be such that $N_{d,\ell}$ is even. We have
\begin{equation*}
M_1(d,n;\ell) = \frac1{2^{N_{d,\ell}-\ell}}  \binom{n+1}{\ell+1}\binom{N_{d,\ell}}{N_{d,\ell}/2}.
\end{equation*}
\end{lemma}

\begin{proof}
We first note that
\begin{equation}\label{eq:goat}
M_1(d,n;\ell) = \frac1{\# \mathbb{B}_{d,n}} \sum_{x \in U_{n,\ell}} \ \sum_{\substack{V \in \mathbb{B}_{d,n} \\ x \in V(\mathbb{Q})}} 1.
\end{equation}
Next, we remark that for any $x \in U_{n,\ell}$ we have 
$$
\sum_{\substack{V \in \mathbb{B}_{d,n} \\ x \in V(\mathbb{Q})}} 1=
2^{N_{d,n}-N_{d,\ell}}\sum_{\substack{V \in \mathbb{B}_{d,\ell} \\ \tilde x \in V(\mathbb{Q})}} 1,
$$
where $\tilde x\in \mathbb{P}^\ell(\QQ)$ is obtained from $x$ by removing its $n-\ell$ zero coordinates. 
Moreover, there exists an obvious bijection $\iota: \mathbb{B}_{d,\ell}\to \mathbb{B}_{d,\ell}$ such that 
$$
\iota \left(\left\{
V \in \mathbb{B}_{d,\ell}: \tilde x \in V(\mathbb{Q})\right\}\right)=
\left\{
V \in \mathbb{B}_{d,\ell}: f_V(1,\dots,1)=0\right\},
$$
where $f_V$ is the degree $d$ form defining $V$.
But $f_V(1,\dots,1)=N_{d,\ell}-2\nu(f_V)$ for any 
$V \in \mathbb{B}_{d,\ell}$, where $\nu$ is given in Definition \ref{def-f}. Hence it follows that 
\begin{equation}\label{eq:3.2}
\begin{split}
\sum_{\substack{V \in \mathbb{B}_{d,n} \\ x \in V(\mathbb{Q})}} 1
&=
2^{N_{d,n}-N_{d,\ell}}
\#\left\{
V \in \mathbb{B}_{d,\ell}: \nu(f_V)=N_{d,\ell}/2\right\}\\
&=2^{N_{d,n}-N_{d,\ell}-1} \binom{N_{d,\ell}}{N_{d,\ell}/2},
\end{split}
\end{equation}
on recalling our assumption that  $N_{d,\ell}$ is even.
The statement of the lemma now follows on inserting this into \eqref{eq:goat} and recalling 
the equalities \eqref{Cardinality Bdn} and \eqref{Cardinality Unl}.
\end{proof}

The  second moment bound is much more complicated. We shall need  the following piece of notation. 
\begin{definition}\label{def-f'}
Let $n\geq 1$
and let $j\in \{1,\dots,n\}.$
Given $f\in \ZZ[x_0,\dots,x_n]$, we define  $\even_j(f)$ to be the part of $f$ that is 
even with  respect to the first $j$ variables. Similarly, we define 
$\odd_j(f)$ to be the part of $f$ that is  odd  with respect to the first $j$ variables. 
Next, let 
$a\in \{1,\dots,n\}$ and  $b\in \{0,\dots,n+1-a\}$.
We also set $R_a^b(f)$ to be the sum of the monomials of $f$ involving only the first $a$ and the final $b$ variables.
\end{definition}

We are now ready to establish the following  upper bound for the second moment.

\begin{lemma}
\label{Lemma second moment}
Let $d, n \geq 2$ and let $\ell \in \{1, \dots, \min\{d,n\}\}$ be such that $N_{d,\ell}$ is even. Then 
\begin{equation*}
\frac{M_2(d,n;\ell)}{M_1(d,n;\ell)^2} \leq 
1 + \boldsymbol{1}_{\ell=n} + O \left( \left(\frac{3}{4}\right)^{\min\{\ell,\mu_\ell\}}+\frac1{\min\{d,n\}}+
 \frac1{M_1(d,n;\ell)}\right) ,
\end{equation*}
where
\begin{equation}\label{eq:soil}
\mu_\ell=\frac{(n-\ell)(\ell+1)}{n+1}.
\end{equation}
\end{lemma}

\begin{proof}
We start by noting that
\begin{equation*}
M_2(d,n;\ell) = \frac1{\# \mathbb{B}_{d,n}} \sum_{x, y \in U_{n,\ell}} \sum_{\substack{V \in \mathbb{B}_{d,n} \\ x, y \in V(\mathbb{Q})}} 1.
\end{equation*}
We set $m_{n,\ell} = \max\{2\ell-n,0\}$ and for $i \in \{m_{n,\ell}, \dots, \ell\}$ and $j \in \{0, \dots, i\}$, we let $T_{n,\ell}^{(i,j)}$ denote the set of pairs of elements $x, y \in U_{n,\ell}$ which have exactly $i+1$ nonzero coordinates in common, and are such that among these $i+1$ coordinates exactly $j$ are different when the final nonzero coordinate shared by $x$ and $y$ is viewed as being $1$. Moreover, we let $T_{n,\ell}$ denote the set of pairs of elements of $ U_{n,\ell}$ which do not have any nonzero coordinates in common. We notice that
\begin{equation*}
U_{n,\ell} \times U_{n,\ell} = \left( \bigsqcup_{i=m_{n,\ell}}^{\ell} \ \bigsqcup_{j=0}^i T_{n,\ell}^{(i,j)} \right) \sqcup T_{n,\ell},
\end{equation*}
which allows us to write
\begin{equation}
\label{Equality M2}
M_2(d,n;\ell) = M_1(d,n;\ell) + \Sigma^{(1)}(d,n;\ell)+\Sigma^{(2)}(d,n;\ell)+\Sigma^{(3)}(d,n;\ell),
\end{equation}
where
\begin{equation*}
\Sigma^{(1)}(d,n;\ell) = \frac1{\# \mathbb{B}_{d,n}} \sum_{j=1}^{\ell} \ \sum_{(x, y) \in T_{n,\ell}^{(\ell,j)}} \sum_{\substack{V \in \mathbb{B}_{d,n} \\ x, y \in V(\mathbb{Q})}} 1,
\end{equation*}
and
\begin{equation*}
\Sigma^{(2)}(d,n;\ell) = \frac1{\# \mathbb{B}_{d,n}} \sum_{i=m_{n,\ell}}^{\ell-1} \ \sum_{j=0}^i \ \sum_{(x, y) \in T_{n,\ell}^{(i,j)}} \sum_{\substack{V \in \mathbb{B}_{d,n} \\ x, y \in V(\mathbb{Q})}} 1,
\end{equation*}
and
\begin{equation*}
\Sigma^{(3)}(d,n;\ell) = \frac1{\# \mathbb{B}_{d,n}} \sum_{(x, y) \in T_{n,\ell}} \sum_{\substack{V \in \mathbb{B}_{d,n} \\ x, y \in V(\mathbb{Q})}} 1.
\end{equation*}
We note  that $\Sigma^{(2)}(d,n;\ell)=0$ if $\ell=n$ and $\Sigma^{(3)}(d,n;\ell)=0$ if $\ell \geq n/2$.
The most difficult sum to estimate is  
$\Sigma^{(2)}(d,n;\ell)$. While it ought to be possible to handle the three sums simultaneously, 
by allowing $i$ to also run over $-1$ and $\ell$ in the definition of $\Sigma^{(2)}(d,n;\ell)$, the treatment is  made easier through the restriction $i\leq \ell-1$, since we can then exploit a non-trivial upper bound for $N_{d,i}/N_{d,\ell}.$

Let $i\in \{m_{n,\ell},\dots,\ell\}$ and let $j\in \{0,\dots,i\}$. If $(x,y)\in T_{n,\ell}^{(i,j)}$ then 
$$
\sum_{\substack{V \in \mathbb{B}_{d,n} \\ x, y \in V(\mathbb{Q})}} 1=
2^{N_{d,n}-N_{d,2\ell-i}} 
\sum_{\substack{V \in \mathbb{B}_{d,2\ell-i} \\ \tilde x, \tilde y \in V(\mathbb{Q})}} 1,
$$
where $\tilde x, \tilde y\in \mathbb{P}^{2\ell-i}(\QQ)$ are obtained from $x$ and $y$ by removing their $n-2\ell+i$ common zeros. There exists an obvious bijection 
$\iota:  \mathbb{B}_{d,2\ell-i}\to \mathbb{B}_{d,2\ell-i}$ such that 
$$
\iota\left(
\left\{ 
V \in \mathbb{B}_{d,2\ell-i}: \tilde x, \tilde y \in V(\mathbb{Q})\right\}\right)=
\left\{ 
V \in \mathbb{B}_{d,2\ell-i} : f_V(\xx_{i,j})=f_V(\yy_{i,j})=0\right\},
$$
where
$$
\xx_{(i,j)} = (\underbrace{\strut 1 , \ldots , 1}_{j} , \underbrace{\strut 1 , \dots , 1}_{i+1-j} , 
\underbrace{\strut 1 , \dots , 1}_{\ell-i} , \underbrace{\strut 0 , \dots , 0}_{\ell-i})
$$
and 
$$
\yy_{(i,j)} = (\underbrace{\strut -1 , \dots , -1}_{j} , \underbrace{\strut 1 , \dots , 1}_{i+1-j} , 
\underbrace{\strut 0 , \dots , 0}_{\ell-i} , \underbrace{\strut 1 , \dots , 1}_{\ell-i}).
$$
Hence for 
 $i\in \{m_{n,\ell},\dots,\ell\}$ and $j\in \{0,\dots,i\}$, we have 
\begin{equation}\label{eq:**}
\sum_{\substack{V \in \mathbb{B}_{d,n} \\ x, y \in V(\mathbb{Q})}} 1=
2^{N_{d,n}-N_{d,2\ell-i}} 
\#\left\{V \in \mathbb{B}_{d,2\ell-i} :
f_V(\xx_{i,j})=f_V(\yy_{i,j})=0\right\},
\end{equation}
for  any $(x,y)\in T_{n,\ell}^{(i,j)}$.

Next, we note that for any $x \in U_{n,\ell}$ and $i \in \{m_{n,\ell}, \dots, \ell\}$, $j \in \{0, \dots, i\}$, we have 
\begin{equation*}
\# \left \{ y \in U_{n,\ell} : (x, y) \in T_{n,\ell}^{(i,j)} \right\} = 2^{\ell-i} \binom{n-\ell}{\ell-i} \binom{\ell+1}{i+1} \binom{i}{j},
\end{equation*}
and
\begin{equation*}
\# \left \{ y \in U_{n,\ell} : (x, y) \in T_{n,\ell} \right\} = 2^{\ell} \binom{n-\ell}{\ell+1}.
\end{equation*}
Recalling the equality \eqref{Cardinality Unl}, we thus deduce that
\begin{equation}
\label{Cardinality Tnlij}
\# T_{n,\ell}^{(i,j)} = 2^{2\ell-i} \binom{n+1}{\ell+1} \binom{n-\ell}{\ell-i} \binom{\ell+1}{i+1} \binom{i}{j},
\end{equation}
and
\begin{equation}
\label{Cardinality Tnl}
\# T_{n,\ell} = 2^{2 \ell} \binom{n+1}{\ell+1} \binom{n-\ell}{\ell+1}.
\end{equation}

\subsubsection*{Analysis of $\Sigma^{(1)}(d,n;\ell)$}

This concerns the case $i=\ell$, so that $2\ell-i=\ell$. 
Given $V\in 
 \mathbb{B}_{d,\ell}$ and $j\in \{1,\dots,\ell\}$, we note that
\begin{align*}
f_V(\xx_{i,j})
&=N_{d,\ell}-2\nu(f_V)\\
&=N_{d,\ell}-2\nu(\even_j(f_V))-2\nu(\odd_j(f_V)),
\end{align*}
in the notation of Definition \ref{def-f'}.
Recall that 
 $c_{d,\ell}(j)$ is defined to be the number of monomials of degree $d$ which are odd with respect to $j$ given variables among $\ell+1$, and that it satisfies  \eqref{Definition c}.
Then we find similarly that
 \begin{align*}
f_V(\yy_{i,j})
&=N_{d,\ell}-
c_{d,\ell}(j)-
2\nu(\even_j(f_V))-\left(c_{d,\ell}(j)-2\nu(\odd_j(f_V))\right)\\
&=N_{d,\ell}-
2c_{d,\ell}(j)-
2\nu(\even_j(f_V))+2\nu(\odd_j(f_V)).
\end{align*}
Hence it follows from \eqref{eq:**} that 
$$
\sum_{\substack{V \in \mathbb{B}_{d,n} \\ x, y \in V(\mathbb{Q})}} 1=
2^{N_{d,n}-N_{d,\ell}} 
\#\left\{V \in \mathbb{B}_{d,\ell} : 
\begin{array}{l}
\nu(\even_j(f_V))=(N_{d,\ell}-c_{d,\ell}(j))/2,\\
\nu(\odd_j(f_V))=c_{d,\ell}(j)/2
\end{array}\right\}.
$$
Since the integer $N_{d,\ell}$ is assumed to be even, we deduce that 
\begin{equation*}
\sum_{\substack{V \in \mathbb{B}_{d,n} \\ x, y \in V(\mathbb{Q})}} 1 =
\begin{cases}
0 & \textrm{if } 2 \nmid c_{d,\ell}(j), \\
\displaystyle{2^{N_{d,n}-N_{d,\ell}-1} \binom{N_{d,\ell}-c_{d,\ell}(j)}{(N_{d,\ell}-c_{d,\ell}(j))/2} \binom{c_{d,\ell}(j)}{c_{d,\ell}(j)/2}} & \textrm{if } 2 \mid c_{d,\ell}(j),
\end{cases}
\end{equation*}
for any  $j \in \{1, \dots, \ell\}$ and any $(x, y) \in T_{n,\ell}^{(\ell,j)}$.
Noting that $0<c_{d,j}(\ell)<N_{d,\ell}$, we may
apply the upper bound \eqref{Upper bound Stirling binom}
 twice to  derive the inequality
\begin{equation*}
\sum_{\substack{V \in \mathbb{B}_{d,n} \\ x, y \in V(\mathbb{Q})}} 1 \leq \frac{2^{N_{d,n}+1}}{\pi N_{d,\ell} } \left( 1-\left|
 1-\frac{2c_{d,\ell}(j)}{N_{d,\ell}}\right|^2\right)^{-1/2}.
\end{equation*}
Since  $\ell \leq d$ by assumption, we 
 see that we are in position to apply Lemma~\ref{Lemma coefficients}. 
Thus 
 $
 \left|
 1-\frac{2c_{d,\ell}(j)}{N_{d,\ell}}\right|\leq 0.5,
 $
and we may deduce that
 \begin{equation*}
\sum_{\substack{V \in \mathbb{B}_{d,n} \\ x, y \in V(\mathbb{Q})}} 1 \leq \frac{2^{N_{d,n}+1}}{\pi N_{d,\ell} } \left( 1+O\left(\left|
 1-\frac{2c_{d,\ell}(j)}{N_{d,\ell}}\right|^2\right)\right).
\end{equation*}
 But then a further  
application of   Lemma~\ref{Lemma coefficients} gives
\begin{equation*}
\sum_{\substack{V \in \mathbb{B}_{d,n} \\ x, y \in V(\mathbb{Q})}} 1 \leq \frac{2^{N_{d,n}+1}}{\pi N_{d,\ell}} \left( 1 + O \left( \left(\frac{1}{4} \right)^{\min \{ j, \ell-j \}} \right) \right).
\end{equation*}
Appealing to the equalities \eqref{Cardinality Bdn} and \eqref{Cardinality Tnlij}, we deduce that
\begin{equation*}
\Sigma^{(1)}(d,n;\ell) \leq \frac{2^{\ell+2}}{\pi N_{d,\ell}} \binom{n+1}{\ell+1} \sum_{j=1}^{\ell} \binom{\ell}{j} \left( 1 + O \left( \left(\frac{1}{4} \right)^{\min \{ j, \ell-j \}} \right) \right).
\end{equation*}
But for any $A\geq 1$ we have 
$$
\sum_{a=0}^A \binom{A}{a}=2^A \quad \text{ and } \quad 
\sum_{a=0}^A \binom{A}{a}\left(\frac{1}{4}\right)^a=\left(\frac{5}{4}\right)^A.
$$
Thus 
\begin{equation*}
\sum_{j=1}^{\ell} \binom{\ell}{j} \left( 1 + O \left( \left(\frac{1}{4} \right)^{\min \{ j, \ell-j \}} \right) \right)=
2^\ell \left(1+ O\left(\left(\frac{5}{8}\right)^\ell\right)\right),
\end{equation*}
which yields  in particular that
\begin{align*}
\Sigma^{(1)}(d,n;\ell) & \leq  \frac{2^{2\ell+2}}{\pi N_{d,\ell}} \binom{n+1}{\ell+1} \left( 1 + O\left(
\left(\frac{3}{4}\right)^\ell\right) \right).
\end{align*}
As a result, on appealing to the estimate \eqref{Estimate Stirling binom} and noting that $N_{d,\ell} \geq  d$, we eventually get
\begin{equation}
\label{Upper bound Sigma1}
\begin{split}
\Sigma^{(1)}(d,n;\ell) 
&\leq \frac1{2^{2N_{d,\ell}-2\ell-1}} \binom{N_{d,\ell}}{N_{d,\ell}/2}^2 \binom{n+1}{\ell+1} \left( 1 + O \left( \left(\frac{3}{4}\right)^\ell+\frac{1}{d}\right) \right)\\
&= M_1(d,n;\ell)^2 
\left( \frac{2}{\binom{n+1}{\ell+1}} + O \left( \left(\frac{3}{4}\right)^\ell+\frac{1}{d}\right) \right),
\end{split}
\end{equation}
where the second line follows from Lemma \ref{Lemma first moment}.

\subsubsection*{Analysis of $\Sigma^{(2)}(d,n;\ell)$}

This concerns the case $i\in \{m_{n,\ell},\dots,\ell-1\}$ and $j\in \{0,\dots,i\}$.
Our starting point is the observation that \eqref{eq:**} holds, 
for  any $(x,y)\in T_{n,\ell}^{(i,j)}$.
Then for any 
$V\in 
 \mathbb{B}_{d,2\ell-i}$ we have 
 \begin{align*}
f_V(\xx_{i,j})
&=N_{d,\ell}-2\nu(R_{\ell+1}^0(f_V))
=
N_{d,\ell}-2\nu(R_{i+1}^0(f_V))-2\nu(X_{\ell,i}(f_V)),
\end{align*}
in the notation of Definition \ref{def-f}, where
$X_{\ell,i}(f_V)=R_{\ell+1}^0(f_V)-R_{i+1}^0(f_V)$.
Note that the polynomial $X_{\ell,i}(f_V)$ only has $N_{d,\ell}-N_{d,i}$ coefficients. 
We conclude that $f_V(\xx_{i,j})=0$ if and only if 
\begin{equation}\label{eq:sun}
\nu(X_{\ell,i}(f_V))=\frac{N_{d,\ell}}{2}-
\nu\left( 
\even_j\left(R_{i+1}^0(f_V)\right)\right)
-\nu\left( 
\odd_j\left(R_{i+1}^0(f_V)\right)\right).
\end{equation}
Similarly, we find that 
 $$
f_V(\yy_{i,j})=N_{d,\ell}-c_{d,\ell}(j)-
2\nu\left( 
\even_j\left(R_{i+1}^{\ell-i}(f_V)\right)\right)
-\left(c_{d,\ell}(j)-
2\nu\left( 
\odd_j\left(R_{i+1}^{\ell-i}(f_V)\right)\right)\right),
$$
in the notation of \eqref{Definition c}.
Thus   $f_V(\yy_{i,j})=0$ if and only if 
\begin{equation}\label{eq:moon}
\begin{split}
\nu\left(\even_j\left( Y_{\ell,i}(f_V)\right)\right)
=~&\frac{N_{d,\ell}}{2}-
c_{d,\ell}(j)-
\nu\left( 
\even_j\left(R_{i+1}^0(f_V)\right)\right)\\
&+\nu\left( 
\odd_j\left(R_{i+1}^0(f_V)\right)\right)
+\nu\left( 
\odd_j\left(Y_{\ell,i}(f_V)\right)\right),
\end{split}
\end{equation}
where 
$Y_{\ell,i}(f_V)=R_{i+1}^{\ell-i}(f_V)-R_{i+1}^0(f_V)$.
Note that the polynomial $Y_{\ell,i}(f_V)$ only involves the 
$\ell+i$
variables
$x_{0},\dots,x_{i}, x_{\ell+1},\dots,x_{2\ell-i}$,
with at least one of the variables $x_{\ell+1},\dots,x_{2\ell-i}$ appearing.

It now follows from \eqref{eq:**} that 
$$
\sum_{\substack{V \in \mathbb{B}_{d,n} \\ x, y \in V(\mathbb{Q})}} 1=
2^{N_{d,n}-N_{d,2\ell-i}} 
\#\left\{V \in \mathbb{B}_{d,2\ell-i} : 
\text{ \eqref{eq:sun} and \eqref{eq:moon} hold}
\right\}.
$$
We put 
\begin{align*}
a&=\nu\left( 
\even_j\left(R_{i+1}^0(f_V)\right)\right)\in \left\{0,\dots,N_{d,i}-c_{d,i}(j)\right\},\\
b&=\nu\left( 
\odd_j\left(R_{i+1}^0(f_V)\right)\right)\in \left\{0,\dots,c_{d,i}(j)\right\},\\
c&=\nu\left( 
\odd_j\left(Y_{\ell,i}(f_V)\right)\right)\in \left\{0,\dots,c_{d,\ell}(j)-c_{d,i}(j)\right\}.
\end{align*}
The total number of positions for these $-1$ coefficients is 
$$
\binom{N_{d,i}-c_{d,i}(j)}{a}
\binom{c_{d,i}(j)}{b}
\binom{c_{d,\ell}(j)-c_{d,i}(j)}{c}.
$$
Moreover, it follows from \eqref{eq:sun} that the number of positions for $-1$ in $X_{\ell,i}(f_V)$
$$
\binom{N_{d,\ell}-N_{d,i}}{N_{d,\ell}/2-(a+b)},
$$
on recalling that  the integer $N_{d,\ell}$ is assumed to be even. Similarly, 
it follows from \eqref{eq:moon} that the number of positions for $-1$ in $Y_{\ell,i}(f_V)$ 
is 
$$
\binom{N_{d,\ell}-N_{d,i} -(c_{d,\ell}(j)-c_{d,i}(j))}{N_{d,\ell}/2-c_{d,\ell}(j)-a+b+c}.
$$
For fixed $a,b,c$, among the $N_{d,i}$ monomials in $x_0,\dots,x_i$ only, $a+b$ coefficients are $-1$, and the others are $+1$. Similarly, the signs are already accounted for among the 
 $N_{d,\ell}-N_{d,i}$
monomials in 
$R_{\ell+1}^0(f_V)-R_{i+1}^0(f_V)$ , and in the $N_{d,\ell}-N_{d,i}$ monomials in 
$R_{i+1}^{\ell-i}(f_V)-R_{i+1}^0(f_V)$. Hence the number of free coefficients is 
$$
N_{d,2\ell-i}-N_{d,i}-2(N_{d,\ell}-N_{d,i})=
N_{d,2\ell-i}-2N_{d,\ell}+N_{d,i}.
$$
In conclusion, for any $j \in \{0, \dots, i\}$ and $(x, y) \in T_{n,\ell}^{(i,j)}$, we deduce that 
\begin{align*}
\sum_{\substack{V \in \mathbb{B}_{d,n} \\ x, y \in V(\mathbb{Q})}} 1 & =  2^{N_{d,n}-2N_{d,\ell}+N_{d,i}-1} 
\hspace{-0.2cm}
\sum_{a=0}^{N_{d,i}-c_{d,i}(j)} \ \sum_{b=0}^{c_{d,i}(j)}\
\sum_{c=0}^{c_{d,\ell}(j)-c_{d,i}(j)}
\hspace{-0.2cm}
 \binom{N_{d,i}-c_{d,i}(j)}{a} \binom{c_{d,i}(j)}{b} \\
&\quad\times \binom{c_{d,\ell}(j)-c_{d,i}(j)}{c}
 \binom{N_{d,\ell}-N_{d,i}}{N_{d,\ell}/2-(a+b)} 
\binom{N_{d,\ell}-N_{d,i} -(c_{d,\ell}(j)-c_{d,i}(j))}{N_{d,\ell}/2-c_{d,\ell}(j)-a+b+c} .
\end{align*}
We apply  \eqref{Identity VDM'} to execute the sum over $c$, finding that 
\begin{align*}
\sum_{\substack{V \in \mathbb{B}_{d,n} \\ x, y \in V(\mathbb{Q})}} 1 =~ & \ 2^{N_{d,n}-2N_{d,\ell}+N_{d,i}-1} \sum_{a=0}^{N_{d,i}-c_{d,i}(j)} \ \sum_{b=0}^{c_{d,i}(j)}\
 \binom{N_{d,i}-c_{d,i}(j)}{a} \binom{c_{d,i}(j)}{b} \\
&\times
 \binom{N_{d,\ell}-N_{d,i}}{N_{d,\ell}/2-(a+b)} 
\binom{N_{d,\ell}-N_{d,i}}{N_{d,\ell}/2-\left(N_{d,i}-c_{d,i}(j)-a+b\right)} .
\end{align*}
Note that 
$$
\binom{N_{d,\ell}-N_{d,i}}{N_{d,\ell}/2-\alpha} 
\leq \binom{N_{d,\ell}-N_{d,i}}{
\lfloor (N_{d,\ell}-N_{d,i})/2\rfloor
},
$$
for any integer $\alpha\geq 0$.
On appealing to  the higher Vandermonde identity
\eqref{Identity VDM2}, we may conclude that 
\begin{align*}
\sum_{\substack{V \in \mathbb{B}_{d,n} \\ x, y \in V(\mathbb{Q})}} 1 \leq ~ & \ 2^{N_{d,n}-2N_{d,\ell}+N_{d,i}-1} 
\binom{N_{d,\ell}-N_{d,i}}{
\lfloor (N_{d,\ell}-N_{d,i})/2\rfloor
}
\binom{N_{d,\ell}}{
N_{d,\ell}/2},
\end{align*}
for each $i\in \{m_{n,\ell},\dots, \ell-1\}$ and  $j \in \{0, \dots, i\}$ and $(x, y) \in T_{n,\ell}^{(i,j)}$.

Observe that
$$
 \sum_{j=0}^i \binom{i}{j}=2^i.
$$
We now bring in the expression \eqref{Cardinality Tnlij} for the 
cardinality of $T_{n,\ell}^{(i,j)}$, together with \eqref{Cardinality Bdn}, in order to deduce that 
$$
\Sigma^{(2)}(d,n;\ell) \leq 
\frac{1}{2^{2N_{d,\ell}-2\ell}}\binom{N_{d,\ell}}{
N_{d,\ell}/2}^2\binom{n+1}{\ell+1}^2\Delta,
$$
where
$$
\Delta
=
\binom{N_{d,\ell}}{
N_{d,\ell}/2}^{-1}\binom{n+1}{\ell+1}^{-1}
 \sum_{i=m_{n,\ell}}^{\ell-1} 
 2^{N_{d,i}} 
 \binom{n-\ell}{\ell-i} \binom{\ell+1}{i+1} 
 \binom{N_{d,\ell}-N_{d,i}}{
\lfloor (N_{d,\ell}-N_{d,i})/2\rfloor
}.
$$
It clearly follows from Lemma \ref{Lemma first moment} that 
\begin{equation}\label{eq:leaf}
\Sigma^{(2)}(d,n;\ell) \leq 
M_1(d,n;\ell)^2
\Delta.
\end{equation}

We now turn to an analysis of $\Delta$, beginning with the upper bound
\begin{equation}\label{eq:root}
\Delta\leq 
\binom{n+1}{\ell+1}^{-1}
 \sum_{i=0}^{\ell-1} 
 \binom{n-\ell}{\ell-i} \binom{\ell+1}{i+1} \Delta_i,
\end{equation}
where
$$
\Delta_i= 2^{N_{d,i}} 
\binom{N_{d,\ell}}{
N_{d,\ell}/2}^{-1}
 \binom{N_{d,\ell}-N_{d,i}}{
\lfloor (N_{d,\ell}-N_{d,i})/2\rfloor
}.
$$
Combining the second bound in \eqref{Upper bound Stirling binom} with 
\eqref{Estimate Stirling binom}, we obtain  
$$
\Delta_i\leq \left(1-\frac{N_{d,i}}{N_{d,\ell}}\right)^{-1/2}\left(1+O\left(\frac{1}{N_{d,i}}\right)\right).
$$
But, on recalling that $i\leq \ell-1$,  we observe that 
$$
\frac{N_{d,i}}{N_{d,\ell}}=\prod_{j=i+1}^\ell \frac{j}{d+j}\leq \frac{1}{2^{\ell-i}},
$$
since $\frac{j}{d+j}\leq 1/2$ for $j\leq \ell\leq d$. Hence
$
\Delta_i\leq 1+O( 2^{-(\ell-i)})$, 
from which it follows that 
$$
\Delta\leq 
\binom{n+1}{\ell+1}^{-1}
 \sum_{i=0}^{\ell-1} 
 \binom{n-\ell}{\ell-i} \binom{\ell+1}{i+1} \left(1+O\left( \frac{1}{2^{\ell-i}}\right)\right)
$$
in \eqref{eq:root}.
On the one hand, we have 
\begin{align*}
 \sum_{i=0}^{\ell-1} \binom{n-\ell}{\ell-i} \binom{\ell+1}{i+1}
 &= 
 \sum_{i=1}^{\ell} \binom{n-\ell}{1+\ell-i} \binom{\ell+1}{i}\\
 &=
 \sum_{i=0}^{\ell+1} \binom{n-\ell}{1+\ell-i} \binom{\ell+1}{i}-1-\binom{n-\ell}{\ell+1}\\
&= \binom{n+1}{\ell+1}-1-\binom{n-\ell}{\ell+1},
\end{align*}
by the  Vandermonde identity \eqref{Identity VDM}.
On the other hand, we see that 
\begin{align*}
\sum_{i=0}^{\ell-1} 
 \binom{n-\ell}{\ell-i} \binom{\ell+1}{i+1} \frac{1}{2^{\ell-i}}
&=
\sum_{i=1}^{\ell} 
 \binom{n-\ell}{\ell+1-i} \binom{\ell+1}{i} \frac{1}{2^{\ell+1-i}}\\
 &\leq 
\sum_{i=0}^{\ell+1} 
 \binom{n-\ell}{\ell+1-i} \binom{\ell+1}{i} \frac{1}{2^{\ell+1-i}}\\
 &=
\sum_{j=0}^{\ell+1} 
 \binom{n-\ell}{j} \binom{\ell+1}{j} \frac{1}{2^{j}}.
\end{align*}
Thus, in the light of Lemma \ref{lem:prob}, we deduce that 
$$
\binom{n+1}{\ell+1}^{-1}
\sum_{i=0}^{\ell-1} 
 \binom{n-\ell}{\ell-i} \binom{\ell+1}{i+1} \frac{1}{2^{\ell-i}}
\leq 
\left(\frac{3}{4}\right)^{\mu_\ell},
$$
where $\mu_\ell$ is given by \eqref{eq:soil}.
Putting this together in \eqref{eq:leaf}, we deduce that 
\begin{equation}\label{eq:stalk}
\begin{split}
\Sigma^{(2)}(d,n;\ell)
\leq~& M_1(d,n;\ell)^2\left(1
-\frac{1}{\binom{n+1}{\ell+1}}
-\frac{\binom{n-\ell}{\ell+1}}{\binom{n+1}{\ell+1}}\right) +O\left( 
\left(\frac{3}{4}\right)^{\mu_\ell}
\right).
\end{split}
\end{equation}

\subsubsection*{Analysis of $\Sigma^{(3)}(d,n;\ell)$}

We remark that reasoning exactly as in the proof of  \eqref{eq:3.2}, we deduce that 
\begin{equation*}
\sum_{\substack{V \in \mathbb{B}_{d,n} \\ x, y \in V(\mathbb{Q})}} 1 = 2^{N_{d,n}-2N_{d,\ell}-1} \binom{N_{d,\ell}}{N_{d,\ell}/2}^2,
\end{equation*}
for any  $(x, y) \in T_{n,\ell}$.
Recalling  \eqref{Cardinality Bdn} and \eqref{Cardinality Tnl}, we thus see that
\begin{equation}
\label{Equality Sigma3}
\begin{split}
\Sigma^{(3)}(d,n;\ell) 
&= \frac1{2^{2N_{d,\ell}-2\ell}} \binom{N_{d,\ell}}{N_{d,\ell}/2}^2 \binom{n+1}{\ell+1} \binom{n-\ell}{\ell+1}\\
&= M_1(d,n;\ell)^2  \frac{\binom{n-\ell}{\ell+1}}{\binom{n+1}{\ell+1}},
\end{split}
\end{equation}
by Lemma~\ref{Lemma first moment}.

\subsubsection*{Conclusion of the argument}

Noting that 
$\binom{n+1}{\ell+1} \geq n$ 
if $\ell < n$, we may now put together \eqref{Upper bound Sigma1}, \eqref{eq:stalk} and \eqref{Equality Sigma3} to eventually obtain
\begin{align*}
\sum_{\nu = 1}^3 \Sigma^{(\nu)}(d,n;\ell) \leq~& M_1(d,n;\ell)^2
\left(1+ \boldsymbol{1}_{\ell=n}
+O\left(\left(\frac{3}{4}\right)^{\min\{\ell,\mu_\ell\}} +\frac{1}{\min\{d,n\}}
\right)\right).
\end{align*}
Recalling the equality \eqref{Equality M2} we see that this completes the proof. 
\end{proof}

\section{Proof of the main results}\label{s:proofs}

Recall the definition \eqref{eq:moment} of the moments $M_k(d,n;\ell)$, for $k\geq 0$.
We are now in position to establish Theorems~\ref{Theorem all d}
and \ref{Theorem almost all d}. Both rely on the inequality
\begin{equation}
\label{Inequality CS}
M_1(d,n;\ell)^2 \leq r_{d,n} M_2(d,n;\ell),
\end{equation}
 for any $\ell \in \{1, \dots, n \}$,
which follows from the Cauchy--Schwarz inequality.

\begin{proof}[Proof of Theorem~\ref{Theorem all d}]
We assume that $d$ is sufficiently large and that $n\geq d+\log d$. Let $\ell \in \{1, \dots, d \}$ be  such that $N_{d,\ell}$ is even. Then, on applying Lemma~\ref{Lemma second moment} and using our assumption $n \geq d+\log d> d$, we obtain
\begin{equation}
\label{Inequality key}
1 \leq r_{d,n} \left( 1 + \boldsymbol{1}_{\ell=n} + O \left( 
\left(\frac{3}{4}\right)^{\min\{\ell,\mu_\ell\}}
+
\frac1{d} 
+\frac1{M_1(d,n;\ell)} \right) \right),
\end{equation}
where
$$
\mu_\ell=\frac{(n-\ell)(\ell+1)}{n+1}.
$$
Note that 
in the special case $d=n=\ell$, a repeated application of  \eqref{Estimate Stirling binom} yields
\begin{align*}
M_1(d,d;d)  \sim \frac{2^{d+1/2}}{\pi^{1/2} N_{d,d}^{1/2}}  
\sim \frac{2^{1/2} d^{1/4}}{\pi^{1/4}},
\end{align*}
whence $M_1(d,d;d)\gg d^{1/4}$.

We start by handling the case where $d$ is a power of $2$. 
The equality \eqref{Equality valuation} implies that $N_{d,d}$ is even, since
$
v_2(N_{d,d})=2s_2(d)-s_2(2d)=1.
$
Taking 
$\ell=d$,  we note that $\mu_d=(n-d)(d+1)/(n+1)$  is a strictly increasing function of $n$. Thus
$$
\mu_d\geq \frac{(d+1)\log d}{d+1+\log d}
\geq \frac{10\log d}{11},
$$
since $n\geq d+\log d$ and $1+(\log d)/(d+1)\leq 11/10$ for large enough values of $d$. 
But then it follows that 
$$
\left(\frac{3}{4}\right)^{\mu_d}\leq \frac{1}{d^{\theta}},
$$
where $\theta=10\log(4/3)/11>1/4$.
Since $r_{d,n}\leq 1$ and $d<n$,  the choice $\ell=d$ in  \eqref{Inequality key} therefore provides
\begin{equation*}
r_{d,n} = 1 + O \left( \frac1{d^{1/4}} + \frac1{M_1(d,n;d)} \right).
\end{equation*}
This  is satisfactory since $M_1(d,n;d) \geq M_1(d,d;d)\gg d^{1/4}$.

We now deal with the case where $d$ is not a power of $2$. We let $2^{k_0(d)}$ and $2^{k_1(d)}$ be  the largest and second largest powers of $2$, respectively,  
appearing in the binary expansion of $d$. It is clear that 
\begin{equation}
\label{Inequality d}
d \leq 1+2+\cdots+2^{k_1(d)-1} +2^{k_1(d)}+2^{k_0(d)}\leq 
2^{k_1(d)+1} + 2^{k_0(d)},
\end{equation}
on evaluating the first $k_1(d)$ terms as a geometric series. 
 Let 
\begin{equation*}
\ell =
\begin{cases}
2^{k_1(d)} & \textrm{if } 2^{k_1(d)} > d/7, \\
2^{k_1(d)} + 2^{k_0(d)-2} & \textrm{if } 2^{k_1(d)} \leq d/7.
\end{cases}
\end{equation*}
This choice ensures that the integers $d$ and $\ell$ share a common digit $1$ in their binary expansions. Indeed, otherwise we would have $2^{k_1(d)} \leq d/7$ and $k_0(d)-2 = k_1(d)$. But then \eqref{Inequality d} 
would yield $d \leq 6 \cdot 2^{k_1(d)} \leq 6d/7$, which is a contradiction.
It therefore follows from Lemma \ref{lem:N-even} that  $N_{d,\ell}$ is even. 

Next we claim that 
\begin{equation}\label{eq:lamp}
\frac{d}{7}<\ell \leq \frac{d}{2}.
\end{equation}
This is obvious when   $2^{k_1(d)} > d/7$, since then  
$\ell =2^{k_1(d)}$ and
$$
3\cdot 2^{k_1(d)}=2^{k_1(d)} +2^{k_1(d)+1}\leq 
2^{k_1(d)} +2^{k_0(d)}\leq d,
$$
which is more than enough to ensure $\ell\leq d/2$.
Suppose that 
$2^{k_1(d)} \leq  d/7$. Then 
$$
\ell=2^{k_1(d)} + 2^{k_0(d)-2}=\frac{3}{4}2^{k_1(d)} +\frac{1}{4}\left(2^{k_1(d)}+2^{k_0(d)}\right) 
\leq \left(\frac{3}{4}\cdot \frac{1}{7} +\frac{1}{4}\right)d \leq \frac{d}{2}.
$$
It follows from  \eqref{Inequality d} 
that $2^{k_0(d)}\geq d-2^{k_1(d)+1}$. Hence
$$
\ell=2^{k_1(d)} + 2^{k_0(d)-2}\geq 2^{k_1(d)} +\frac{1}{4}\left(d-2^{k_1(d)+1}\right) 
\geq \frac{d}{4}>\frac{d}{7}.
$$
This completes the verification of \eqref{eq:lamp}.

For $n\geq d+\log d> d$, we see that 
$$
\mu_\ell=\frac{(n-\ell)(\ell+1)}{n+1}\geq \frac{(d-\ell)(\ell+1)}{d+1},
$$
since the left hand side is  a strictly increasing function of $n$. 
Applying \eqref{eq:lamp}, we may set
\begin{equation}\label{eq:goaty}
\lambda = \frac{\ell}{d} \in  \left(\frac{1}{7}, \frac{1}{2}\right].
\end{equation}
The inequality  $\mu_\ell\geq d/10$ happens if 
$10(1-\lambda)d\lambda +10(1-\lambda)\geq d+1$.
But $10(1-\lambda)\geq 10\cdot 1/2\geq 1$, by \eqref{eq:goaty}.  
Hence it suffices to check that 
$10(1-\lambda)\lambda \geq 1$, which is true for all $\lambda$ in the range \eqref{eq:goaty}.
Thus it follows that 
$\min\{\ell,\mu_\ell\}\geq d/10$. 

Note that 
$(3/4)^{d/10}<1/d$   for sufficiently large $d$.
Appealing to the inequality \eqref{Inequality key} and noting that $r_{d,n}\leq 1$,
 we finally obtain
\begin{equation}
\label{Estimate rdn}
r_{d,n} = 1 + O \left( \frac1{d} +\frac1{M_1(d,n;\ell)} \right).
\end{equation}
Adopting the notation \eqref{eq:goaty},  
it follows from \eqref{Estimate Stirling binom} and 
Lemma~\ref{Lemma first moment}  that
\begin{equation*}
M_1(d,n;\ell) \gg 2^{\lambda d} \binom{(1+\lambda)d}{d}^{-1/2} \binom{n}{\lambda d}.
\end{equation*}
Appealing to Stirling's formula \eqref{Stirling} and using our assumption $n \geq  d+\log d \geq d$, we thus get 
\begin{align*}
M_1(d,n;\ell) & \gg 2^{\lambda d} \frac{d^{d/2+1/4}(\lambda d)^{\lambda d/2+1/4}}{((1+\lambda)d)^{(1+\lambda)d/2+1/4}} \cdot \frac{d^{ d+1/2}}{(\lambda d)^{\lambda d+1/2} ((1-\lambda) d)^{(1-\lambda) d+1/2}} \\
& \gg \left( \frac{2^{\lambda} }{\lambda^{\lambda/2} (1+\lambda)^{(1+\lambda)/2} (1-\lambda)^{1-\lambda}} \right)^d \frac1{d^{1/4}}.
\end{align*}
A numerical check now shows that 
\begin{equation*}
\frac{2^{\lambda} }{\lambda^{\lambda/2} (1+\lambda)^{(1+\lambda)/2} (1-\lambda)^{1-\lambda}} > 1,
\end{equation*}
for any $\lambda \in (1/7, 1/2]$. 
Recalling the estimate \eqref{Estimate rdn}, this therefore  completes the proof of Theorem~\ref{Theorem all d}.
\end{proof}

\begin{proof}[Proof of Theorem~\ref{Theorem almost all d}]

Throughout the proof we may assume that 
$n \geq d^{1/2}
\log d$. 
We take
\begin{equation*}
\mathcal{D} = \left\{ d \geq 1 : v_2(d) \leq \log \log \log d \right\},
\end{equation*}
and we  start by noting that
\begin{equation*}
\# \{ d \leq X : d \notin \mathcal{D} \} \ll \frac{X}{\log \log X}.
\end{equation*}
Thus $\mathcal{D}$ has density $1$, as required for
Theorem~\ref{Theorem almost all d}.

For any $d \in \mathcal{D}$,
 we now assume that $\ell$ is a positive integer 
such that $ \ell < d^{1/5}$  and  $N_{d, \ell}$ is even.
In particular $\ell<n$, since $n \geq d^{1/2}
\log d$. Combining the inequality \eqref{Inequality CS} and Lemma~\ref{Lemma second moment}, we therefore deduce that
\begin{equation}
\label{Lower bound r}
1 \leq  r_{d,n} + O \left( \frac1{\sqrt{d}}+\left(\frac{3}{4}\right)^{\min\{\ell,\mu_\ell\}} +\frac1{M_1(d,n;\ell)} \right) ,
\end{equation}
where $\mu_\ell=(n-\ell)(\ell+1)/(n+1)$. Note that 
$\mu_\ell\geq \ell$ if and only if $n\geq \ell^2+2\ell$.
But this is obviously satisfied if  $n \geq d^{1/2}
\log d$ and $\ell<d^{1/5}$, assuming that $d$ is sufficiently large.

We next focus on a lower bound for $M_1(d,n;\ell)$.
Writing that $2^{\ell} \geq \ell+1$, 
it follows from \eqref{Estimate Stirling binom} and 
Lemma~\ref{Lemma first moment}  that
\begin{align*}
M_1(d,n;\ell) & \gg \frac{n}{N_{d,\ell}^{1/2}} \binom{n}{\ell} 
 = \frac{n}{(\ell!)^{1/2}} \cdot \frac{n!}{(n-\ell)!} \cdot \frac{(d!)^{1/2}}{((d+\ell)!)^{1/2}}.
\end{align*}
Appealing to Stirling's formula \eqref{Stirling}, we deduce that
\begin{align*}
M_1(d,n;\ell) & \gg \frac{n}{\ell^{\ell/2+1/4}
} \cdot \frac{n^{n+1/2}}{(n-\ell)^{n-\ell+1/2} } \cdot \frac{d^{d/2+1/4}}{(d+\ell)^{(d+\ell)/2+1/4}} \\
& \gg 
\frac{n}{\ell^{\ell/2+1/4}}
\left( \frac{n}{d^{1/2}} \right)^{\ell} \left(1-\frac{\ell}{n} \right)^{-n+\ell}  \left(1+\frac{\ell}{d} \right)^{-(d+\ell)/2}.
\end{align*}
Using our assumptions $\ell < d^{1/5}$ and $n \geq d^{1/2}\log d$,  we claim that 
\begin{equation*}
\left(1-\frac{\ell}{n} \right)^{-n+\ell} \left(1+\frac{\ell}{d} \right)^{-(d+\ell)/2} \gg e^{\ell/2}.
\end{equation*}
Taking logarithms of both sides, we note that this  is equivalent to proving 
$$
(-n+\ell) \log\left(1-\frac{\ell}{n}\right) - \frac{d+\ell}{2} \log\left(1+\frac{\ell}{d}\right) \geq \frac{\ell}{2}+O(1).
$$
But $\log(1-\ell/n)=-\ell/n+O(\ell^2/n^2)$ and 
$\log(1+\ell/d)=\ell/d+O(\ell^2/d^2)$. Hence the claim follows under our assumption that 
$\ell < d^{1/5}$ and $n \geq d^{1/2}\log d$.
Using  our assumption $n \geq d^{1/2}\log d$, we finally obtain
\begin{equation*}
M_1(d,n;\ell) \gg 
\frac{n 
(\log d)^{\ell}
e^{\ell/2}}{\ell^{\ell/2+1/4}}\gg 
\frac{d^{1/2}
(\log d)^\ell}{\ell^{\ell/2}}.
\end{equation*}
Returning to 
\eqref{Lower bound r} and observing that $r_{d,n}\leq 1$, we therefore obtain
\begin{equation}\label{eq:shone}
r_{d,n} =1+ O \left( \frac1{\sqrt{d}}+\left(\frac{3}{4}\right)^\ell+
\frac{\ell^{\ell/2}}{d^{1/2}
{(\log d)^{\ell}}}
 \right) .
\end{equation}

It is now time to choose a suitable value of $\ell$. 
We need to ensure that $N_{d,\ell}$ is even, that $\ell<d^{1/5}$ and that the error terms in \eqref{eq:shone} are satisfactory for 
Theorem~\ref{Theorem almost all d}.
We shall assume that $d$ is sufficiently large. 
Let 
\begin{equation*}
\ell = 2^{v_2(d)} + 
2^{\lfloor \frac{\log\log d}{\log 2} \rfloor+3}.
\end{equation*}
Since 
$$
v_2(d)\leq \log\log\log d<
\left\lfloor \frac{\log\log d}{\log 2} \right\rfloor+3
$$ 
for sufficiently large $d\in \mathcal{D}$, 
so it follows that 
$d$ and $\ell$ share a common digit $1$  in their binary expansion. Hence 
Lemma \ref{lem:N-even} ensures that 
$N_{d,\ell}$ is even. Moreover 
$$
\ell \leq 2^{\log\log\log d} + 2^{\log\log d/\log 2+3} \leq \log\log d + 8\log d
$$
and $\ell\geq 2^{\log\log d/\log2+2}=4\log d$.
It follows that $4\log d\leq \ell \leq 9\log d$ for 
sufficiently large $d\in \mathcal{D}$, which is more than enough to ensure that
 $\ell<d^{1/5}$.

Turning to the three error terms in 
\eqref{eq:shone}, the first is   clearly satisfactory. The second is satisfactory on observing that 
$4\log(4/3)> 1/2$.
To handle the third, we note that 
\begin{align*}
\frac{\ell^{\ell/2}}{d^{1/2}(\log d)^\ell} =
\frac{1}{d^{1/2}} \left(\frac{\ell^{1/2}}{\log d}\right)^\ell
&\leq 
\frac{1}{d^{1/2}} \left(\frac{3}{\sqrt{\log d}}\right)^\ell,
\end{align*}
on taking $\ell \leq 9\log d$. This is at most $1/\sqrt{d}$ for 
sufficiently large $d\in \mathcal{D}$, 
which therefore  completes the proof of Theorem~\ref{Theorem almost all d}.
\end{proof}

\end{document}